\documentclass{article}

\usepackage{amsmath}
\usepackage{amssymb}
\usepackage{amsthm}
\usepackage[cmtip,all]{xy}
\usepackage{mathrsfs}  
\usepackage{enumerate}
\usepackage{tikz-cd}
\usepackage{breqn}
\usepackage[utf8]{inputenc}
\usepackage{stackrel}
\usepackage{hyperref}
\usepackage{aliascnt}

\newtheorem{theorem}{Theorem}[section]

\newaliascnt{lemma}{theorem}
\newtheorem{lemma}[lemma]{Lemma}
\aliascntresetthe{lemma}

\newaliascnt{proposition}{theorem}
\newtheorem{proposition}[proposition]{Proposition}
\aliascntresetthe{proposition}

\theoremstyle{definition}

\newtheorem*{notation}{Notation}

\newaliascnt{definition}{theorem}
\newtheorem{definition}[definition]{Definition}
\aliascntresetthe{definition}

\theoremstyle{remark}
\newtheorem{remark}[theorem]{Remark}

\numberwithin{equation}{section}

\newtheorem*{ack}{Acknowledgement}
\newcommand{\OO}{\mathcal{O}} 
\newcommand{\tor}{\mathrm{tor}}

\newcommand{\rsw}{\operatorname{rsw}}
\newcommand{\sw}{\operatorname{Sw}}

\newcommand{\Spec}{\operatorname{Spec}}
\newcommand{\N}{\mathbb{N}}
\newcommand{\Z}{\mathbb{Z}}
\newcommand{\Q}{\mathbb{Q}}
\newcommand{\R}{\mathbb{R}}
\newcommand{\m}{\mathfrak{m}}
\newcommand{\pole}{\mathrm{pole}}
\newcommand{\ab}{\mathrm{ab}}
\begin{document}

\title{Generalized Hasse-Herbrand functions in positive characteristic}

\author{Isabel Leal\thanks{Department of Mathematics, University of Chicago, 5734 S. University Avenue, Chicago, IL, 60637, USA. Electronic address: \texttt{isabel@math.uchicago.edu}}}

\date{}

\maketitle

\begin{abstract}
	Let $L/K$ be an extension of complete discrete valuation fields of positive characteristic, and assume that the residue field of $K$ is perfect. The residue field of $L$ is not assumed to be perfect.

	In this paper, we show that the generalized Hasse-Herbrand function $\psi_{L/K}^\ab$ has properties similar to those of its classical counterpart. In particular, we prove that $\psi_{L/K}^\ab$ is continuous, piecewise linear, increasing, convex, and satisfies certain integrality properties.
\end{abstract}

\section{Introduction}

From classical ramification theory, if $L/K$ is a finite Galois extension of local fields and $G=G(L/K)$ is the Galois group of this extension, we know that there are lower and upper ramification subgroups $G_{t}$ and $G^t$ of $G$, where $t\in [0,\infty)$, related to each other by the classical Hasse-Herbrand $\psi$-function: 
\[
G^t = G_{\psi_{L/K}(t)}.
\]

In \cite{isabel2}, we  considered an extension $L/K$ of complete discrete valuation fields, where the residue field of $K$ is perfect and of positive characteristic, but the residue field of $L$ is possibly \textit{imperfect}.  We defined, for such an extension, a generalized Hasse-Herbrand $\psi$-function $\psi_{L/K}^\ab$, and showed that it coincides with the classical $\psi_{L/K}$ when $L/K$ is a finite Galois extension of local fields.  
 
In this paper, we study these generalized Hasse-Herbrand $\psi$-functions for fields of positive characteristic, and show that they have properties similar to those of the classical $\psi$-function. In particular, we prove that $\psi_{L/K}^{\ab}$ is continuous, piecewise linear, increasing, convex, and satisfies certain integrality properties. 

More precisely, we prove the following theorems:
\begin{theorem}[\autoref{theorem:cont} and \autoref{theorem:piecewiselinear}]
	Let $L/K$ be an extension of complete discrete valuation fields of positive characteristic. Assume that the residue field of $K$ is perfect. Then $\psi_{L/K}^{\ab}:[0,\infty)\to [0,\infty)$ is continuous, piecewise linear, increasing, and convex. 
\end{theorem}
\begin{theorem}[\autoref{theorem:integrality}] Let $L/K$ be an extension of complete discrete valuation fields of positive characteristic. Assume that the residue field of $K$ is perfect.	Then: \begin{enumerate}[(i)]
		%\item At each point, the right and left derivatives of $\psi_{L/K}^{\ab}(t)$ are integers. 
		\item For $t\in \Z_{\geq 0}$, we have $\psi_{L/K}^{\ab}(t)\in\Z_{\geq0}$.	
		\item For  $t\in\Q_{\geq0}$, we have $\psi_{L/K}^{\ab}(t)\in\Q_{\geq0}$. 
		\item The right and left derivatives of $\psi_{L/K}^{\ab}$ are integer-valued.
	\end{enumerate} 
\end{theorem}

Furthermore, we compute $\psi_{L/K}^\ab$ for certain extensions $L/K$ of complete discrete valuation fields of characteristic $p>0$. Assume that the residue field of $K$ is perfect.  When $e(L/K)$ is prime to $p$, we show that 
\[\psi_{L/K}^\ab(t)=e(L/K)t,\]
where $e(L/K)$ denotes the ramification index of $L/K$.  

More interestingly, we consider the case where there exists a field $L_0$ such that $L/L_0$ is a separable totally ramified cyclic extension of degree $p$, and $e(L_0/K)=1$.  In \autoref{theorem:degreep} we prove that, in this case,
	\[
	\psi^{\ab}_{L/K}(t) = \begin{cases}
	t, & t \leq \dfrac{\delta_{\tor}(L/K)}{p-1} \\ pt - \delta_\tor(L/K), & t> \dfrac{\delta_{\tor}(L/K)}{p-1}
	\end{cases}
	\]
	where $\delta_\tor(L/K)$ is the length of the torsion part of the completed $\OO_L$-module of relative differential forms with log poles $\hat{\Omega}^1_{\OO_L/\OO_K}(\log)$. This formula is a generalization of the classical formula obtained in \cite[Chapter V, \S3]{serre1979local}.
	
\begin{notation}
	Through this paper, for a complete discrete valuation field $K$, $\OO_K$ denotes its ring of integers, $\m_K$ the maximal ideal,  $\pi_K$ a prime element, and $G_K$ the absolute Galois group. Lowercase $k$  denotes the residue field of $K$, and $v_K$ the discrete valuation.	When we say that $K$ is a local field, we mean that  $K$ is a complete discrete valuation field with perfect (not necessarily finite) residue field.
	
	We write \[\hat{\Omega}^1_{\OO_K}(\log)= \varprojlim\limits_m \Omega^1_{\OO_K}(\log)/\m_K^m \Omega^1_{\OO_K}(\log),\] 
	where   
	\[
	\Omega^1_{\OO_K}(\log)= (\Omega_{\OO_K}^1\oplus(\OO_K\otimes_{\Z}K^\times))/(da - a\otimes a, \,  a \in \OO_K, \, a \neq 0).
	\]
	By completed free $\OO_L$-module with basis $\{e_\lambda\}_{\lambda\in\Lambda}$, we mean $\varprojlim\limits_m M/\m_L^mM$, where $M$ is the free $\OO_L$-module with basis $\{e_\lambda \}_{\lambda\in\Lambda}$.
	
	We shall denote by $P_\tor$ the torsion part of an abelian group $P$. Let $L/K$ an extension of complete discrete valuation fields of positive characteristic $p>0$. Throughout this paper, $e(L/K)$ shall denote the ramification index of $L/K$ and, when $K$ is of characteristic zero, $e_K$ shall denote the absolute ramification index of $K$. When $k$ is perfect and $L/K$ is separable,  $\delta_\tor(L/K)$ shall denote the length of $\left(\dfrac{\hat{\Omega}^1_{\OO_L}(\log)}{\OO_L\otimes_{\OO_K}\hat{\Omega}^1_{\OO_K}(\log)}\right)_\tor$.

	Following the notation in \cite{kato1989swan}, we write, for $A$ a ring over $\Q$ or a smooth ring over a field of characteristic $p>0$, and $n\neq 0$ possibly divisible by $p$, 
	\[ H_n^q(A)=   H^q((\Spec A)_{\mathrm{et}}, \Z/n\Z(q-1)  )\]
	and
	\[H^q(A)= \varinjlim\limits_n   H_n^q(A).\]
\end{notation}
\section{Definition and previous results}

Through  this section, let $L/K$ be an %eparable
extension of  complete discrete valuation fields such that the
residue field of $K$ is perfect and 
of characteristic $p>0$. In \cite{isabel2}, we defined   
generalizations $\psi_{L/K}^{\mathrm{AS}}$ and $\psi_{L/K}^\ab$ of the classical $\psi$-function for this case, in the sense that they both coincide with the classical $\psi_{L/K}$ when $L/K$ is a finite Galois extension of local fields (\cite[Theorem 5.5]{isabel2}).

In this section, we review the definition of $\psi_{L/K}^{\ab}$. For a detailed discussion, we refer to \cite{isabel2}.
Assume first that the residue field $k$ of $K$ is algebraically closed. 
For $t \in \Z_{(p)}$, $t\geq0$, define $\psi^{\mathrm{ab}}_{L/K}(t)\in \R_{\geq 0}$ as 
\begin{align*}
&\psi^{\ab}_{L/K}(t)= \\ &\inf\left\lbrace s\in \Z_{(p)} \, \middle| \, \begin{array}{l@{}l@{}}
\operatorname{Im}(F_{e(K'/K)t}H^1(K') \to H^1(L')) \subset F_{e(L'/L)s}H^1(L')   \\  \text{for all finite, tame extensions $K'/K$ of complete discrete} \\ \text{valuation fields such that }  e(L'/L)s, e(K'/K)t\in \Z  \end{array}  \right\rbrace,
\end{align*}
where $L'=LK'$, and $F_nH^1(K)$ denotes Kato's ramification filtration defined in \cite{kato1989swan}. Extend $\psi^{\ab}_{L/K}$ to $\R_{\geq 0}$ by putting 
\[
\psi^{\ab}_{L/K}(t)=\sup\{\psi^{\ab}_{L/K}(s):s\leq t, s \in \Z_{(p)} \}.
\]
When $k$ is not necessarily algebraically closed, define  $\psi^{\ab}_{L/K}=\psi_{\widehat{LK_{ur}}/\widehat{K_{ur}}}^\ab$.

In a similar way, we can define $\varphi^\ab_{L/K}$, which is shown to be the inverse of $\psi_{L/K}^\ab$ when the latter is bijective (\cite[Proposition 5.1]{isabel2}).  
Further, we have established a formula for $\psi_{L/K}^\ab(t)$ for sufficiently large $t \in \R_{\geq 0}$, which is given by the following theorem: 
\begin{theorem}[{\cite[Theorem 5.4]{isabel2}}] \label{theorem:oldformula}
	Let $L/K$ be a separable extension of complete discrete valuation fields. Assume that $K$ has perfect residue field of characteristic $p>0$. Let $t\in\R_{\geq 0}$ be such that
	\[\begin{cases}
	t\geq \dfrac{2e_K}{p-1}+\dfrac{1}{e(L/K)}+\left\lceil\dfrac{\delta_\tor(L/K)}{e(L/K)}\right\rceil & \text{if $K$ is of characteristic $0$,} \\[.5cm] t> \dfrac{p}{p-1} \dfrac{\delta_\tor(L/K)}{e(L/K)} & \text{if $K$ is of characteristic $p$.}
	\end{cases}\]
	Then 
	\[\psi^{\mathrm{ab}}_{L/K}(t)= e(L/K)t-\delta_{\tor}(L/K). \]
\end{theorem} 

We shall review some concepts that were necessary for obtaining \autoref{theorem:oldformula} and will be used through the rest of this paper.  For more detailed background on Kato's Swan conductor, we recommend \cite{kato1989swan}. For an overview of classical ramification theory and modern advances, we suggest \cite{xiao2015ramification}.

Let $L$ be a complete discrete valuation field of characteristic $p>0$. Let $l$ be the residue field of $L$ and write $L=l((\pi_L))$ for some  prime $\pi_L\in L$. Let $\{b_\lambda\}_{\lambda\in \Lambda}$ be a lift of a  $p$-basis of $l$ to $\OO_L$. Then $\hat{\Omega}^1_{\OO_L}(\log)$ is the completed free $\OO_L$-module with basis $\{db_\lambda, d\log \pi_L:\lambda\in\Lambda\}$.  Write $\hat{\Omega}^1_L=L\otimes_{\OO_L}\hat{\Omega}^1_{\OO_L}(\log)$. 

Denote by $W_s(L)$ the Witt vectors of length $s$. There is a homomorphism $d:W_s(L)\to \hat{\Omega}_{L}^1$ given by
\[a=(a_{s-1},\ldots, a_0)\mapsto  \sum_i a_i^{p^i-1}da_i.\]
\begin{remark} 
	In the literature, the operator $d:W_s(L)\to \hat{\Omega}_{L}^1$ is often denoted by $F^{s-1}d$. 
\end{remark}
Define, for  $\omega \in 
\hat{\Omega}_{L}^1$ and $a \in W_s(L)$, 
\[
v_L^{\log} (\omega) = \sup\left\{n: \omega \in \pi_L^n\otimes_{\OO_L}\hat{\Omega}_{\OO_L}^1(\log)\right\},
\]
and
\[
v_L(a)=-\max_i\{- p^i v_L(a_i)\}=\min_i\{p^i v_L(a_i)\}.
\]
Then introduce filtrations of   $\hat{\Omega}_{L}^1$ and $W_s(L)$ by the subgroups 
\[
F_n \hat{\Omega}_{L}^1 = \{\omega \in \hat{\Omega}_{L}^1:v_L^{\log} (\omega) \geq -n\}
\]
and
\[
F_n W_s(L) = \{a \in W_s(L):v_L(a) \geq -n\},
\]
respectively, where $n\in \Z_{\geq0}$. The latter filtration was defined by Brylinski in \cite{Brylinski1983}. 

By the theory of Artin-Schreier-Witt, there are isomorphisms
\[
W_s(L)/(F-1)W_s(L)\simeq H^1(L,\Z/p^s \Z),
\] 
where $F$ is the endomorphism of Frobenius. Kato defined in \cite{kato1989swan} the filtration $F_n H^1(L, \Z / p^s\Z)$
%, $n\in\Z_{\geq 0}$,
as the image of $F_n W_s(L)$ under this map. We recall that, for $\chi \in  H^1(L, \Z / p^s\Z)$, the Swan conductor $\sw \chi$ is the smallest $n$ such that $\chi \in F_n H^1(L, \Z / p^s\Z)$. 
\begin{definition} \label{def best}
	Let	$a \in W_s(L)$, and $n$ be the smallest non-negative integer such that $a \in F_n W_s(L)$. 
	We say that $a$ is best if there is no $a'\in W_s(L)$ mapping to the same element as $a$ in $H^1(L, \Z / p^s\Z)$  such that $a'\in F_{n'}W_s(L)$ for some  non-negative integer $n'<n$.
\end{definition}
When $v_L(a)\geq0$, $a$ is clearly best. When $v_L(a)<0$, $a$ is best if and only if  there are no $a', b \in W_s(L)$ satisfying
\[
a= a' + (F-1)b
\]	
and  $v_L(a)< v_L(a')$.

Observe that  $a\in F_n W_s(L)\backslash F_{n-1} W_s(L)$ is best if and only if $n= \sw \chi$, where $\chi$ is the image of $a$ under $F_nW_s(L)\to H^1(L,\Z/p^s\Z)$. We remark that ``best $a$'' is not unique.

We review the refined Swan conductor defined in \cite{isabel2}, which shall be necessary for our proofs:
\begin{proposition}[{\cite[Proposition 2.8]{isabel2}}] 	\label{prop:besta}	Let $L$ be a complete discrete valuation field of characteristic $p>0$.  
	\begin{enumerate}[(i)] 
		\item There is a unique homomorphism
		\[
		\rsw: F_n H^1(L, \Z / p^s \Z) \to F_n\hat{\Omega}^1_L/F_{\lfloor n/p\rfloor}\hat{\Omega}^1_L,
		\]
		called refined Swan conductor, such that the composition 
		\begin{center}
			\begin{tikzcd}
				F_n W_s(L) \arrow[r] & F_n H^1(L, \Z / p^s \Z) \arrow[r] &   F_n\hat{\Omega}^1_L/F_{\lfloor n/p\rfloor}\hat{\Omega}^1_L
			\end{tikzcd}
		\end{center}
		coincides with
		\[
		d:  F_n W_s(L)  \to  F_n\hat{\Omega}^1_L/F_{\lfloor n/p\rfloor}\hat{\Omega}^1_L.
		\]
		\item For $\lfloor n/p \rfloor \leq m \leq n$, the induced map 
		\[
		\rsw:  F_n H^1(L, \Z / p^s \Z)/ F_m H^1(L, \Z / p^s \Z) \to F_n\hat{\Omega}^1_L/F_m\hat{\Omega}^1_L
		\]
		is injective. 
		
		In particular, for $\chi \in  F_n H^1(L, \Z / p^s \Z)$, if $\rsw\chi \in F_n\hat{\Omega}^1_L/F_m\hat{\Omega}^1_L$ is non-trivial and the class of $\omega \in F_n\hat{\Omega}^1_L$, then $\sw \chi =  -v_L^\log(\omega)$. 
	\end{enumerate}
\end{proposition}
\begin{remark}
		Our refined Swan conductor $\rsw$ is a refinement of the refined Swan conductor defined by Kato in \cite[\S 5] {kato1989swan}. Related results were obtained by  Yatagawa in \cite{yatagawa2016equality}, where the author compares the non-logarithmic filtrations of Matsuda (\cite{matsuda1997swan}) and Abbes-Saito (\cite{abbessaito}) in positive characteristic.
\end{remark}

\section{Fundamental properties}\label{sec:psitrans}

In this section, we prove new properties of $\psi_{L/K}^\ab$ in the positive characteristic case. The central idea is reduce the understanding of $\psi_{L/K}^\ab$ to the case of complete discrete valuation fields with perfect residue fields. This is achieved by studying the behavior of the Swan conductor after taking certain extensions of $L$ with perfect residue field.

\begin{remark}The author would like to thank Kazuya Kato for pointing out that ramification and reduction to the perfect residue field case were also studied by  Borger in \cite{Borger2004}. 
\end{remark} We start with the following lemmas:
\begin{lemma} \label{lemma:commdia}
	Let $L'/L$ be an extension of  complete discrete valuation fields of characteristic $p>0$, and write $e=e(L'/L)$. We have the following commutative diagram: 
	\begin{center}
		\begin{tikzcd}
			F_n H^1(L, \Z / p^s \Z)/ F_{\lfloor\frac{n}{p}\rfloor} H^1(L, \Z / p^s \Z)  \arrow[r]\arrow[d]&  F_n\hat{\Omega}^1_L/ F_{\lfloor\frac{n}{p}\rfloor}\hat{\Omega}^1_L \arrow[d]\\ F_{en} H^1(L', \Z / p^s \Z)/ F_{\lfloor\frac{en}{p}\rfloor} H^1(L', \Z / p^s \Z)  \arrow[r] &  F_{en}\hat{\Omega}^1_{L'}/F_{\lfloor\frac{en}{p}\rfloor}\hat{\Omega}^1_{L'}
		\end{tikzcd}
	\end{center}
\end{lemma}
\begin{proof}
	The central point is to observe that the vertical arrows are well-defined, which follows from $e\lfloor\frac{n}{p}\rfloor\leq \lfloor\frac{en}{p}\rfloor$.
\end{proof}
\begin{lemma} \label{lemma:Lpi}
	Let $L$ be a complete discrete valuation field of characteristic $p>0$,  $\{T_\lambda : \lambda \in \Lambda\}$ be a lift of a $p$-basis of $l$ to $\OO_L$, and  \[L_\pi = \bigcup_{\lambda \in \Lambda} \bigcup_{j\in \N}L\left(T_\lambda^{1/p^j}\right).\] 
	Let $\chi \in 	F_n H^1(L, \Z / p^s \Z)\setminus F_{n-1} H^1(L, \Z / p^s \Z)$, and write
	\[
		\rsw \chi = a d\log \pi_L + \sum_\lambda b_\lambda dT_\lambda.
	\] 
	Assume that $-v_L(a)= n>0$. Then
	\[
		\sw \chi_{\widehat{L_\pi}} = \sw \chi, 
	\] 
	where $\chi_{\widehat{L_\pi}}$ is the image of $\chi$ in $H^1(\widehat{L_\pi}, \Z / p^s \Z)$.
%	\[
%		F_n H^1(L, \Z / p^s \Z) \to \dfrac{F_n H^1(L, \Z / p^s \Z)}{F_{\lfloor\frac{n}{p}\rfloor} H^1(L, \Z / p^s \Z)} \to 
%		\dfrac{F_{n} H^1(\widehat{L_\pi}, \Z / p^s \Z)}{F_{\lfloor\frac{n}{p}\rfloor} H^1(\widehat{L_\pi}, \Z / p^s \Z)} 
%	\]
%	is non-trivial.
\end{lemma}
\begin{proof}
%	We have a commutative diagram 
%	\begin{center}
%		\begin{tikzcd}
%			F_n H^1(L, \Z / p^s \Z)/ F_{\lfloor\frac{n}{p}\rfloor} H^1(L, \Z / p^s \Z)  \arrow[r]\arrow[d]&  F_n\hat{\Omega}^1_L/ F_{\lfloor\frac{n}{p}\rfloor}\hat{\Omega}^1_L \arrow[d]\\ F_{n} H^1(\widehat{L_\pi}, \Z / p^s \Z)/ F_{\lfloor\frac{n}{p}\rfloor} H^1(\widehat{L_\pi}, \Z / p^s \Z)  \arrow[r] &  F_{n}\hat{\Omega}^1_{\widehat{L_\pi}}/F_{\lfloor\frac{n}{p}\rfloor}\hat{\Omega}^1_{\widehat{L_\pi}}
%		\end{tikzcd}
%	\end{center}
	Taking $L'= \widehat{L_\pi}$, we have a commutative diagram 
	given by \autoref{lemma:commdia}. Further, by the assumption, the image of the refined Swan conductor $\rsw\chi$ in $F_{n}\hat{\Omega}^1_{\widehat{L_\pi}}/F_{\lfloor\frac{n}{p}\rfloor}\hat{\Omega}^1_{\widehat{L_\pi}}$ is non-trivial. More precisely, the image is 
	\[
		\rsw \chi_{\widehat{L_\pi}} = a d\log \pi_{L}= a d\log \pi_{\widehat{L_\pi}}.
	\] 
	Thus the result follows from \autoref{prop:besta} (ii). 
\end{proof}
\begin{lemma} \label{lemma:Llambda}
	Let $L$ be a complete discrete valuation field of characteristic $p>0$, $m_i = p^i$ for $i\in\Z_{\geq 2}$, and $\{T_\lambda : \lambda \in \Lambda\}$ be a lift of a $p$-basis of $l$ to $\OO_L$. For $\lambda \in \Lambda$, write 
	\[
	L_{\lambda,i} = \left( \bigcup_{\gamma \in \Lambda \setminus \{\lambda\}} \bigcup_{j\in \N}L\left(T_\gamma^{1/p^j}\right) \right)\cup \left(\bigcup_{j\in N} L\left(\pi_L^{1/m_i}\right)\left(\left(T_\lambda-\pi_L^{1/m_i}\right)^{1/p^j}\right)\right).
	\] 
	Let $\chi \in 	F_n H^1(L, \Z / p^s \Z)\setminus F_{n-1} H^1(L, \Z / p^s \Z)$, and write
	\[
	\rsw \chi = a d\log \pi_L + \sum_\gamma b_\gamma dT_\gamma.
	\] 
	Assume that $-v_L(b_\lambda)= n>0$. Then
	\[
	\sw \chi_{\widehat{L_{\lambda,i}}} = m_i \sw \chi_L - 1,
	\] 
	where $\chi_{\widehat{L_{\lambda,i}}}$ is the image of $\chi$ in $H^1(\widehat{L_{\lambda,i}}, \Z / p^s \Z)$.
\end{lemma}
\begin{proof}
	The central idea is similar to that of \autoref{lemma:Lpi}. 
	Taking $L'= \widehat{L_{\lambda,i}}$, we have a commutative diagram given by \autoref{lemma:commdia},
%	\begin{center}
%		\begin{tikzcd}
%			F_n H^1(L, \Z / p^s \Z)/ F_{\lfloor\frac{n}{p}\rfloor} H^1(L, \Z / p^s \Z)  \arrow[r]\arrow[d]&  F_n\hat{\Omega}^1_L/ F_{\lfloor\frac{n}{p}\rfloor}\hat{\Omega}^1_L \arrow[d]\\ F_{m_i n} H^1(\widehat{L_{\lambda,i}}, \Z / p^s \Z)/ F_{\lfloor\frac{m_i n}{p}\rfloor} H^1(\widehat{L_{\lambda,i}}, \Z / p^s \Z)  \arrow[r] &  F_{m_i n}\hat{\Omega}^1_{\widehat{L_{\lambda,i}}}/F_{\lfloor\frac{m_i n}{p}\rfloor}\hat{\Omega}^1_{\widehat{L_{\lambda,i}}}
%		\end{tikzcd}
%	\end{center}
	 and, by the assumption, the image of the refined Swan conductor $\rsw\chi$ in $F_{m_i n}\hat{\Omega}^1_{\widehat{L_{\lambda,i}}}/F_{\lfloor\frac{m_i n}{p}\rfloor}\hat{\Omega}^1_{\widehat{L_{\lambda,i}}}$ is non-trivial. To see that, let $\omega \in F_n\hat{\Omega}^1_L$ be given by
	\[\omega = a d\log \pi_L + \sum_\gamma b_\lambda dT_\gamma. \]
	Then the image $\omega'$ of $\omega$ in $\hat{\Omega}^1_{\widehat{L_{\lambda,i}}}$ is $a d\log \pi_L + b_\lambda dT_\lambda$. Further, in $\hat{\Omega}^1_{\widehat{L_{\lambda,i}}}$, \[dT_\lambda = d \pi_L^{1/m_i},\] so we can write 
	\[
	\omega' = a \dfrac{d \left((\pi_L^{1/m_i})^{m_i}\right)}{(\pi_L^{1/m_i})^{m_i}} + b_\lambda \pi_L^{1/m_i} \dfrac{d \pi_L^{1/m_i}}{\pi_L^{1/m_K}} = b_\lambda \pi_L^{1/m_i} d\log(\pi_L^{1/m_i}).
	\] 
	
	Since $n= - v_L(b_\lambda)\geq -v_L(b_\gamma)$ for $\gamma \in \Lambda$ and $n \geq -v_L(a)$, we have 
	\[
		-v_{\widehat{L_{\lambda,i}}}^\log(\omega')= -v_{\widehat{L_{\lambda,i}}}(b_\lambda\pi_L^{1/m_i}) = -m_i v_L(b_\lambda)-1 = m_i n -1.
	\]
	Observe that 
	\[
		m_i n - 1 - \left\lfloor\dfrac{m_in}{p}\right\rfloor = p^in-1 - p^{i-1}n = p^{i-1}n(p-1) -1 >0,
	\]
	so the image of $\rsw\chi$ in $F_{m_i n}\hat{\Omega}^1_{\widehat{L_{\lambda,i}}}/F_{\lfloor\frac{m_i n}{p}\rfloor}\hat{\Omega}^1_{\widehat{L_{\lambda,i}}}$ is non-trivial. More precisely, we have 
	\[
		\rsw \chi_{\widehat{L_{\lambda,i}}} =  b_\lambda \pi_L^{1/m_i} d\log(\pi_L^{1/m_i})
	\]
	and, from \autoref{prop:besta} (ii),
	\[
	\sw \chi_{\widehat{L_{\lambda,i}}} = m_i \sw \chi_L - 1. \qedhere
	\]
\end{proof}
\begin{proposition} \label{proposition:supremum}
	Let $L/K$ be an extension of complete discrete valuation fields of characteristic $p>0$. Assume that $k$ is perfect. Then \[\psi_{L/K}^{\ab}=\sup\left\{\dfrac{\psi^{\ab}_{M/K}}{e(M/L)}: M \in S\right\},\] where $S$ is the set of all extensions of complete discrete valuation fields $M/L/K$ such that  the residue field of $M$ is perfect.
\end{proposition}
\begin{proof}
	%It is sufficient to find a set of complete discrete valuation fields with perfect residue fields $S$ such that $\sw \chi_{L'} = \sup\{\sw\chi_{M'}/e(M'/L'):M\in S\}$ for any $\chi \in H^1(K',\Z/p^s\Z)$, where $K'/K$ is a finite, tame extension, $L'=LK'$, and $M'=MK'$. 
	
	Let $\{T_\lambda: \lambda \in \Lambda\}$ be a lift of a $p$-basis of $l$ to $\OO_L$ and $L_\pi$ and $L_{\lambda,i}$ be as in \autoref{lemma:Lpi} and \autoref{lemma:Llambda}, respectively. Put 
	\[S_L= \left\{\widehat{L_\pi}\right\}\cup \left\{\widehat{L_{\lambda,i}} : \lambda \in \Lambda,i \in \N \right\},\]
	%Denote $\widehat{L_{\lambda,i}}'=\widehat{L_{\lambda,i}}K'$ and $\widehat{L_\pi}' =\widehat{L_\pi}K'$. 
	and let $\chi \in H^1(K,\Z/p^s\Z)$, 
	\[\rsw \chi_{L} = a d\log\pi_{L} + \sum_{\lambda}b_\lambda dT_\lambda \in F_n\hat{\Omega}^1_{L}/F_{\lfloor \frac{n}{p}\rfloor} \hat{\Omega}^1_{L}\] 
	be the refined Swan conductor of $\chi_{L}$, and $n=\sw \chi_{L}>0$. We have that $n= \max \left\{\{-v_L(a)\}\cup \{-v_L({b_\lambda}): \lambda \in \Lambda\}\right\}$. From \autoref{lemma:Lpi}, we have  \[\rsw \chi_{\widehat{L_\pi}}= a d\log\pi_L\]
	and $\sw \chi_L \geq \sw \chi_{\widehat{L_\pi}}$. From \autoref{lemma:Llambda}, we have, when $-v_L(a)=n$, \[\sw \chi_L = \sw \chi_{\widehat{L_\pi}}.\]
	On the other hand, when $n= - v_L(b_\lambda)$, by \autoref{lemma:Llambda}, for $i\geq 2$,
	\[
	\sw \chi_{\widehat{L_{\lambda,i}}} = m_i \sw \chi_L - 1,
	\]
so
\[ \sw \chi_L = \dfrac{\sw \chi_{\widehat{L_{\lambda,i}}}}{m_i} + \dfrac{1}{m_i}.\]
Therefore, in this case, we have
\[
	\sw \chi_L = \lim_{i\to \infty} \dfrac{\sw  \chi_{\widehat{L_{\lambda,i}}}}{m_i}= \sup_{i}\left\{ \dfrac{\sw  \chi_{\widehat{L_{\lambda,i}}}}{m_i}\right\}.
\]

From this it follows that, if $t=\sw \chi$, 
\[
\sw \chi_L =  \sup_{M\in S_L}\left\{ \dfrac{\psi^\ab_{M/K}(t)}{e(M/L)}\right\}= \sup_{M\in S}\left\{ \dfrac{\psi^\ab_{M/K}(t)}{e(M/L)}\right\}.
\]

Let $K'/K$ be a finite tame extension, and $L'=LK'$. Let $\chi \in H^1(K',\Z/p^s\Z)$ be such that $\sw \chi = e(K'/K)t$, where $t\in\Z_{(p)}$. From the same argument as before, we have that
\begin{align*}
\dfrac{\sw \chi_{L'}}{e(L'/L)} &=  \sup_{M\in S_{L'}}\left\{ \dfrac{\psi^\ab_{M/K'}(e(K'/K)t)}{e(M/L)}\right\}\\&=\sup_{M\in S'}\left\{ \dfrac{\psi^\ab_{M/K}(t)}{e(M/L)}\right\}\leq \sup_{M\in S}\left\{ \dfrac{\psi^\ab_{M/K}(t)}{e(M/L)}\right\},
\end{align*}
where $S'$ consists of all extensions of complete discrete valuation fields $M/L'/K'$ such that  the residue field of $M$ is perfect.
Since $\psi^\ab_{L/K}\geq \frac{\psi^\ab_{M/K}}{e(M/K)}$ for every $M\in S$, we conclude that 
%From this it follows that \[\psi_{L/K}^{\ab}\leq\sup\left\{\dfrac{\psi^{\ab}_{M/K}}{e(M/L)}: M \in S\right\}.\]
%On the other hand, 
% Therefore we conclude that 
\[\psi_{L/K}^{\ab}=\sup\left\{\dfrac{\psi^{\ab}_{M/K}}{e(M/L)}: M \in S\right\}. \qedhere\]
\end{proof}

Assume that the residue field $k$ of $K$ is perfect, and $L/K$ is  a (possibly transcendental) extension of complete discrete valuation fields of positive characteristic. In what follows, we will reduce the study of the properties of $\psi^{\ab}_{L/K}$ to the case where $L/K$ is a finite extension of local fields. In this case, $\psi^{\ab}_{L/K}$ is almost fully understood by the classical theory (\cite{serre1973course}). However, the classical $\psi$-function is only defined for separable extensions, while we want to allow the presence of inseparability. For that reason, we shall first we investigate $\psi^{\ab}_{L/K}$ for purely inseparable finite extensions of local fields, and show that in this case $\psi^{\ab}_{L/K}$ is the identity map. From now on, we will write $\psi_{L/K}$ for $\psi^{\ab}_{L/K}$ in the case of a finite extension of local fields, and call this function ``classical'' (even when the extension $L/K$ has inseparability degree greater than $1$). 
\begin{lemma} Let $L/K$ be a finite extension of local fields of positive characteristic. Assume that $L/K$ is purely inseparable. Then 
	\[
	\psi_{L/K}(t) = t.
	\]
\end{lemma}
\begin{proof}
	Since $L/K$ is purely inseparable, we can write $L = K^{1/p^r}$ for some $r\geq 0$.
	The map $N_{L/K}: L \to K$ is given by $x \mapsto x^{p^r}$ and is an isomorphism of complete discrete valuation fields. It induces an isomorphism $G_K \simeq G_L$ which respects upper ramification subgroups, so it follows that $\sw \chi_L = \sw \chi$ for any $\chi \in H^1(K)$.   
%	
%	 By \autoref{lemma:transitivity}, it is enough to consider the extension $L/K$ where $K$ is a local field of positive characteristic and $L=K(\sqrt[p]{\pi_K})$. Let $\chi \in F_n H^1(K, \Z/p\Z)$ be such that $\sw \chi =n\geq 1$ and $\chi$ corresponds to $a\in K$ under Artin-Schreier theory. Assume that $a$ is best in $K$. Since the residue field of $K$ is perfect, $p\nmid n$, and we can write \[a= (\sqrt[p]{a})^p =(\sqrt[p]{a})^p-(\sqrt[p]{a}) + (\sqrt[p]{a}) \]
%	in $L$.  We claim that $\sqrt[p]{a}$ is best. Indeed, if not, then we can write $\sqrt[p]{a} = b^p - b + \tilde{a}$ in $L$, with $v_L(\sqrt[p]{a})<v_L(\tilde{a})$. This implies \[v_L(b)= (1/p)v_L(\sqrt[p]{a})=(1/p^2)v_L(a) = \frac{v_K(a)}{p},\] which is impossible since $p\nmid n$.  Hence $\sqrt[p]{a}$ is best and $\sw \chi_L = n$. 
	The same argument holds for the extension $LK'/K'$, where $K'/K$ is a finite tamely ramified extension, so we have that 
	\[
	\psi_{L/K}(t)=t.\qedhere
	\]
\end{proof}

We now have all the necessary tools to start the reduction to the classical case. We start by proving the following transitivity formula: 
\begin{lemma} \label{lemma:transitivity}
	Let $L/K$ be an extension of  complete discrete valuation fields of characteristic $p>0$, and assume that the residue fields of $K$ and $L$ are perfect. Then, if $M/L$ is a finite extension of complete discrete valuation fields, we have  
	\[
	\psi^{\ab}_{M/K} = \psi^{\ab}_{M/L} \circ \psi^{\ab}_{L/K}=\psi_{M/L} \circ \psi^{\ab}_{L/K}.
	\]
\end{lemma}
\begin{proof}
%	\begin{align*}
%	&\psi^{\ab}_{L/K}(t)= \\ &\inf\left\lbrace s\in \Z_{(p)} \, \middle| \, \begin{array}{l@{}l@{}}
%	\operatorname{Im}(F_{e(K'/K)t}H^1(K') \to H^1(LK')) \subset F_{e(LK'/L)s}H^1(LK')   \\  \text{for all finite, tame extensions $K'/K$ of complete discrete} \\ \text{valuation fields such that }  e(LK'/L)s, e(K'/K)t\in \Z  \end{array}  \right\rbrace,
%	\end{align*}
%	
	From the definition, we can assume that $l$ and $k$ are algebraically closed. It is enough to show, for $t\in\Z_{(p)}$, $t\geq0$, that $\psi^{\ab}_{M/K}(t) \leq \psi^{\ab}_{M/L} \circ \psi^{\ab}_{L/K}(t)$ and 
	$\psi^{\ab}_{M/K}(t) \geq \psi^{\ab}_{M/L} \circ \psi^{\ab}_{L/K}(t)$.
	
	 Let  $s\in \Z_{(p)}$ such that $s> \psi^{\ab}_{M/L}(\psi^{\ab}_{L/K}(t))$. Recall that $\psi^{\ab}_{M/L}$ coincides with the classical $\psi$-function (\cite{isabel2}). Therefore we can write $s=  \psi^{\ab}_{M/L}(\psi^{\ab}_{L/K}(t)+\epsilon)$ for some $\epsilon>0$. Then, for every $s'\in \Z_{(p)}$ such that $0\leq s'\leq\psi^{\ab}_{L/K}(t)+\epsilon$, we have
	\[
		\operatorname{Im}(F_{e(L'/L)s'}H^1(L') \to H^1(ML')) \subset F_{e(ML'/M)s}H^1(ML') 
	\]
	for all finite, tame extensions $L'/L$ of complete discrete valuation fields such that $e(ML'/M)s, e(L'/L)s'\in \Z$. On the other hand, given  $r\in\Z_{(p)}$ such that $r\geq \psi^{\ab}_{L/K}(t)$,
	\[
		\operatorname{Im}(F_{e(K'/K)t}H^1(K') \to H^1(LK')) \subset F_{e(LK'/L)r}H^1(LK')  
	\]
	for all finite, tame extensions $K'/K$ of complete discrete valuation fields such that $e(LK'/L)r, e(K'/K)t\in \Z$. Taking $r\in \Z_{(p)}$ such that $\psi^{\ab}_{L/K}(t)+\epsilon \geq r\geq \psi^{\ab}_{L/K}(t)$, we see 
	that
	\[
	\operatorname{Im}(F_{e(K'/K)t}H^1(K') \to H^1(LK')) \subset F_{e(MK'/M)s}H^1(LK')  
	\]
	for all finite, tame extensions $K'/K$ of complete discrete valuation fields such that $e(ML'/M)s, e(K'/K)t\in \Z$. Hence
	 $\psi^{\ab}_{M/K}(t) \leq \psi^{\ab}_{M/L} \circ \psi^{\ab}_{L/K}(t)$.
	
	Now let $s\in \Z_{(p)}$ such that $0\leq s< \psi^{\ab}_{M/L}(\psi^{\ab}_{L/K}(t))$. By the properties  of the classical $\psi$, we can write $s=  \psi^{\ab}_{M/L}(\psi^{\ab}_{L/K}(t)-\epsilon)-\epsilon'$ for some $\epsilon, \epsilon'>0$. We can assume that $(\psi^{\ab}_{L/K}(t)-\epsilon)\in \Z_{(p)}$ by making an appropriate choice of $\epsilon, \epsilon'$. Then we have, from the definitions, 
	\[
	\operatorname{Im}(F_{e(K'/K)t}H^1(K') \to H^1(LK')) \not\subset F_{e(LK'/L)(\psi^{\ab}_{L/K}(t)-\epsilon)}H^1(LK')  
	\]
	and
	\[
	\operatorname{Im}(F_{e(LK'/L)(\psi^{\ab}_{L/K}(t)-\epsilon)}H^1(LK') \to H^1(MK')) \not\subset F_{e(MK'/M)s}H^1(MK') 
	\]
	for some finite, tame extension $K'/K$ of complete discrete valuation fields such that $e(MK'/M)s, e(LK'/L)(\psi^{\ab}_{L/K}(t)-\epsilon), e(K'/K)t\in \Z$. It follows that $\psi^{\ab}_{M/K}(t)>s$. Thus $\psi^{\ab}_{M/K}(t) \geq \psi^{\ab}_{M/L} \circ \psi^{\ab}_{L/K}(t)$.
\end{proof}
\begin{lemma} \label{lemma:perfectpsi}
	Let $L/K$ be an extension of complete discrete valuation fields of characteristic $p>0$. Assume that $k$ and $l$ are perfect. Let  $L_0 = K \otimes_{W(k)}W(l)$. Then 
	\[
		\psi^{\ab}_{L/K} = \psi^{\ab}_{L/L_0}= \psi_{L/L_0}.
	\]
%	%where $K\subset L_0 \subset L$ 
%	and $L/L_0$ is a finite extension of local fields. 
\end{lemma}
\begin{proof}
	$L_0$ has perfect residue field $l$ and $e(L_0/K)=1$. We have that $\psi^{\ab}_{L_0/K}$ is the identity map,  and $L/L_0$ is a finite extension of local fields. Then the result follows from \autoref{lemma:transitivity}.
\end{proof}

\begin{theorem} \label{theorem:cont}
	Let $L/K$ be an extension of complete discrete valuation fields of characteristic $p>0$. Assume that $k$ is perfect. Then $\psi_{L/K}^{\ab}:[0,\infty)\to [0,\infty)$ is continuous, increasing, and convex. In particular, it is bijective.
\end{theorem}
\begin{proof}
	Let $S$ be as in \autoref{proposition:supremum}, and \[\mathcal{S}=\left\{\dfrac{\psi^{\ab}_{M/K}}{e(M/L)}: M \in S\right\}.\]
	By \cite[Chapter IV, \S 3, Proposition 13]{serre1979local}, the classical Hasse-Herbrand $\psi$-functions are convex. Then, by \autoref{lemma:perfectpsi}, $\mathcal{S}$ is an equicontinuous set of convex functions. It follows from \autoref{proposition:supremum} that $\psi_{L/K}^{\ab}$ is continuous and convex.  
	
	To prove that $\psi_{L/K}^{\ab}$ is increasing, observe that, since  $\psi_{M/K}^{\ab}(0)=0$ for every $M \in S$, we have that $\psi_{L/K}^{\ab}(0)=0$.  For $t \in (0,1)$ and $y\in (0,\infty)$ we have, by convexity of $\psi_{L/K}^{\ab}$,
	\[
		t \psi_{L/K}^{\ab}(y)=\psi_{L/K}^{\ab}(0)(1-t) + t \psi_{L/K}^{\ab}(y) \geq \psi_{L/K}^{\ab}(t y).
	\] 
	Further, for $y>0$, clearly $\psi_{L/K}^{\ab}(y)>0$. Then, putting $t=x/y$ for $x,y \in (0,\infty)$ with $x<y$, we get
	\[
		\psi_{L/K}^{\ab}(y)> t \psi_{L/K}^{\ab}(y)\geq \psi_{L/K}^{\ab}(x),
	\] 
	so $\psi_{L/K}^{\ab}$ is increasing.
\end{proof}

\begin{lemma}	\label{lemma:integrality} Let $L/K$ be an extension of complete discrete valuation fields of characteristic $p>0$. Assume that $k$ is perfect. %and $[l:l^p]<\infty$. 
	Then, for $t\in \Z_{\geq 0}$, we have $\psi_{L/K}^{\ab}(t)\in\Z_{\geq0}$.	
\end{lemma}
\begin{proof}
	Observe that \[\inf\{s\in \Z_{\geq 0}: \operatorname{Im}(F_tH^1(K) \to H^1(L)) \subset F_{s}H^1(L)  \} \in \Z_{\geq 0}. \]
	Let $m=\inf\{s\in \Z_{\geq 0}: \operatorname{Im}(F_tH^1(K) \to H^1(L)) \subset F_{s}H^1(L)$. Then $m \leq \psi_{L/K}^\ab(t)$. It is enough to show that, for a finite tame extension $K'/K$, 
	\[\inf\{s\in \Z_{\geq 0}: \operatorname{Im}(F_{e(K'/K)t}H^1(K') \to H^1(L')) \subset F_{e(L'/L)m}H^1(L')  \}, \]
	where $L'=LK'$. 
	
	From \autoref{proposition:supremum}, we have, for $\chi \in F_tH^1(K,\Z/p^s\Z)$,  \[\sw \chi_L=\sup\left\{\dfrac{\sw \chi_M}{e(M/L)}: M \in S\right\}=m,\] where $S$ is the set of all extensions of complete discrete valuation fields $M/L/K$ such that  the residue field of $M$ is perfect. For $\chi \in F_{e(K'/K)t}H^1(K',\Z/p^s\Z)$,  
	\begin{align*}
	\dfrac{\sw \chi_{L'}}{e(L'/L)} &=  \sup_{M\in S'}\left\{ \dfrac{\psi^\ab_{M/K'}(e(K'/K)t)}{e(M/L)}\right\}\\&=\sup_{M\in S'}\left\{ \dfrac{\psi^\ab_{M/K}(t)}{e(M/L)}\right\}\\&\leq \sup_{M\in S}\left\{ \dfrac{\psi^\ab_{M/K}(t)}{e(M/L)}\right\} = m,
	\end{align*}
	where $S'$ consists of all extensions of complete discrete valuation fields $M/L'/K'$ such that  the residue field of $M$ is perfect. Thus $\psi_{L/K}^\ab(t)\in \Z$ and, furthermore, 
	\[
	\psi^\ab_{L/K}(t)= \inf\{s\in \Z_{\geq 0}: \operatorname{Im}(F_tH^1(K) \to H^1(L)) \subset F_{s}H^1(L)  \}. \qedhere
	\]
\end{proof}

\begin{theorem} \label{theorem:piecewiselinear}
		Let $L/K$ be an extension of complete discrete valuation fields of characteristic $p>0$. Assume that $k$ is perfect. %and $[l:l^p]<\infty$. 
		Then $\psi_{L/K}^{\ab}:[0,\infty)\to [0,\infty)$ is piecewise linear. 
\end{theorem}
\begin{proof}
	First observe that, from \autoref{theorem:oldformula}, there is $\tilde{t}\in \R_{\geq 0}$ sufficiently large such that $\psi_{L/K}^\ab$ is linear for $t>\tilde{t}$,
	so it is enough to prove that $\psi_{L/K}^\ab$ is piecewise linear on the closed interval $[0,\tilde{t}]$. 
	
	Let $T_\lambda$, $L_\pi$, and $L_{\lambda,i}$ be as in \autoref{lemma:Lpi} and \autoref{lemma:Llambda}.  Denote $\psi^{\ab}_\pi = \psi^{\ab}_{\widehat{L_\pi}/K}$ and $\psi_{\lambda,i}^{\ab} = \psi^{\ab}_{\widehat{L_{\lambda,i}}/K}/e(\widehat{L_{\lambda,i}}/L)$. Let 
	\[f_i = \sup\left(\{\psi_{\lambda,i}^{\ab}: \lambda \in \Lambda\}\cup \{\psi^{\ab}_\pi\}\right).\] 
	
	Here we go over an outline of the proof. The key idea is to show that,  for sufficiently large $i$, there is a cover of $\left[0,\tilde{t}\right]$ by open intervals such that, on each open interval, $f_i$ is piecewise linear, which, by compactness, 
	%it is sufficient to show that there is an open cover of $\left[0,\frac{p}{p-1}\frac{\delta_\tor(L/K)}{e(L/K)}\right]$ such that $\psi_{L/K}^\ab$ is piecewise linear in each open set of the cover.  
	implies that $f_i$ is piecewise linear. 
	Further, we shall show that, on the subintervals where $f_i$ is linear, $f_{i+1}$ is also linear. From $f_i\leq f_{i+1}$, we shall deduce that $\psi_{L/K}^{\ab} = \sup_i \{f_i\}$ is piecewise linear. 
	
	We start with a refinement of the argument used in the proof of \autoref{proposition:supremum}. 	
	Let $\chi \in H^1(K', \Z/p^\Z)$ where $K'/K$ is a finite tame extension. For a field $M\supset K$, we shall denote $M'=MK'$. 
	Write \[\rsw \chi_{L'} = a d\log\pi_{L'} + \sum_{\lambda}b_\lambda dT_\lambda \in F_n\hat{\Omega}^1_L/F_{\lfloor \frac{n}{p}\rfloor} \hat{\Omega}^1_L,\]
	where $n=\sw \chi_{L'}$.  Using the same arguments from \autoref{lemma:Lpi}, \autoref{lemma:Llambda}, and \autoref{proposition:supremum}, we have that, if  $\sw \chi_{L'} = -v_{L'}(a)$,
	  \[\rsw \chi_{\widehat{L_\pi}'}= a d\log\pi_{L'}\]
	and, if  $\sw \chi_{L'} = -v_{L'}(b_{\lambda})$ and $i$ is sufficiently large,
	\[
	\rsw \chi_{\widehat{L_{\lambda,i}}'} =  b_\lambda \pi_{L'}^{e(L'/L)/m_i} d\log(\pi_{L'}^{1/m_i}).
	\]
	From this we get:
	\begin{enumerate}[(i)]
		\item If $\sw \chi_{L'} = -v_{L'}(a)$, then \[\sw \chi_{L'}=\sw \chi_{\widehat{L_{\pi}}'}.\]
		\item If  $\sw \chi_{L'} = -v_{L'}(b_{\lambda})$, and $i$ is sufficiently large, then
		\[
			\dfrac{\sw\chi_{\widehat{L_{\lambda,i}}'}}{e(\widehat{L_{\lambda,i}}'/L')} = \sw \chi_{L'} - \dfrac{e(L'/L)}{e(\widehat{L_{\lambda,i}}'/L')} 
		\]
		so
		\[
			\dfrac{1}{e(\widehat{L_{\lambda,i}}'/\widehat{L_{\lambda,i}})}\dfrac{\sw\chi_{\widehat{L_{\lambda,i}}'}}{e(\widehat{L_{\lambda,i}}/L)} = \dfrac{\sw \chi_{L'}}{e(L'/L)} - \dfrac{1}{e(\widehat{L_{\lambda,i}}/L)}. 
		\]
	\end{enumerate} 
	Therefore, for each $t_0\in \R_{\geq 0}$,  either 
	\[
	\dfrac{\psi^\ab_{\widehat{L_{\lambda,i}}/K}(t_0)}{e(\widehat{L_{\lambda,i}}/L)} = \psi^\ab_{L/K}(t_0) - \dfrac{1}{e(\widehat{L_{\lambda,i}}/L)} 
	\]
	for some $\lambda \in \Lambda$ and all sufficiently large $i$ or 
	\[
		\psi^\ab_{\widehat{L_{\pi}}/K}(t_0) = \psi^\ab_{L/K}(t_0). 
	\]
	Observe that these equalities can hold for only finitely many $\lambda \in \Lambda$. Further, for $t_0\in \Z_{(p)}$, $0<t_0< \tilde{t}$, let $K'/K$ be a finite tame extension such that $e(K'/K)t_0\in \Z$. If 
	\[
		\dfrac{\psi^\ab_{\widehat{L_{\lambda,i}}/K}(t_0)}{e(\widehat{L_{\lambda,i}}/L)} < \psi^\ab_{L/K}(t_0) - \dfrac{1}{e(\widehat{L_{\lambda,i}}/L)},
	\]
	then
	\[
		\psi^\ab_{\widehat{L_{\lambda,i}}'/K'}(e(K'/K)t_0) < e(\widehat{L_{\lambda,i}}/L) \psi^\ab_{L'/K'}(e(K'/K)t_0) - e(L'/L). 
	\]
	From \autoref{lemma:integrality},
	\[
		\psi^\ab_{\widehat{L_{\lambda,i}}'/K'}(e(K'/K)t_0) \leq e(\widehat{L_{\lambda,i}}/L) \psi^\ab_{L'/K'}(e(K'/K)t_0) - e(L'/L) -1, 
	\]
	so it follows that 
	\begin{align*}\dfrac{\psi^\ab_{\widehat{L_{\lambda,i}}/K}(t_0)}{e(\widehat{L_{\lambda,i}}/L)} & \leq \psi^\ab_{L/K}(t_0)- \dfrac{1}{e(\widehat{L_{\lambda,i}}/L)} - \dfrac{1}{e(L'/L)e(\widehat{L_{\lambda,i}}/L)}\\ & \leq \psi^\ab_{L/K}(t_0) - \dfrac{1}{e(\widehat{L_{\lambda,i}}/L)} - \dfrac{1}{e(K'/K)e(\widehat{L_{\lambda,i}}/L)}. \end{align*}
	Since the slopes of $\psi^\ab_{\widehat{L_{\lambda,i}}/K}$ are bounded, it follows that there exists an open interval containing $t_0$ such that, for sufficiently large $i$, 
	\[f_i = \max\left(\{\psi_{\lambda,i}^{\ab}: \lambda \in \Lambda'\}\cup \{\psi^{\ab}_\pi\}\right)\]
	on this interval, where $\Lambda' \subset \Lambda$ is finite. Therefore $f_i$ is piecewise linear on this interval. By compacteness, for sufficiently large $i$,   $f_i$ is piecewise linear on $\left[0,\tilde{t}\right]$. 
	
	To conclude the argument observe that, on an interval where 
	\[
	f_i(t)= \dfrac{\psi^\ab_{\widehat{L_{\lambda,i}}/K}(t)}{e(\widehat{L_{\lambda,i}}/L)} = \psi^\ab_{L/K}(t) - \dfrac{1}{e(\widehat{L_{\lambda,i}}/L)},
	\]
	(respectively, $f_i(t)=\psi^\ab_{\widehat{L_{\pi}}/K}(t) = \psi^\ab_{L/K}(t)$), we also have 
	\[
	\dfrac{\psi^\ab_{\widehat{L_{\lambda,i+1}}/K}(t)}{e(\widehat{L_{\lambda,i+1}}/L)} = \psi^\ab_{L/K}(t) - \dfrac{1}{e(\widehat{L_{\lambda,i+1}}/L)} 
	\]
	(respectively, $f_{i+1}(t)=\psi^\ab_{\widehat{L_{\pi}}/K}(t) = \psi^\ab_{L/K}(t)$)
	for every $\gamma\in\Lambda$.
	On intervals where $f_i(t)= \psi^\ab_{L/K}(t) - 1/e(\widehat{L_{\lambda,i}}/L)$, we have that $\psi_{\lambda,i}^{\ab}$ and $\psi_{\lambda,i+1}^{\ab}$ are linear on the same subintervals, so $f_i$ and $f_{i+1}$ are linear on the same subintervals. From $f_i\leq f_{i+1}$, we get the result.
\end{proof}

\begin{theorem}	\label{theorem:integrality} Let $L/K$ be an extension of complete discrete valuation fields of positive characteristic. Assume that the residue field of $K$ is perfect.	Then: \begin{enumerate}[(i)]
		%\item At each point, the right and left derivatives of $\psi_{L/K}^{\ab}(t)$ are integers. 
		\item For $t\in \Z_{\geq 0}$, we have $\psi_{L/K}^{\ab}(t)\in\Z_{\geq0}$.	
		\item For  $t\in\Q_{\geq0}$, we have $\psi_{L/K}^{\ab}(t)\in\Q_{\geq0}$. 
		\item The right and left derivatives of $\psi_{L/K}^{\ab}$ are integer-valued.
	\end{enumerate} 
\end{theorem}
\begin{proof}
	(i) is given by \autoref{lemma:integrality}. To prove (iii), it is sufficient to show that the slopes of $\psi_{L/K}^\ab$ on intervals where $\psi_{L/K}^\ab$ is linear are integers. Let $I$ be an open interval where $\psi_{L/K}^\ab$ is linear and $\psi_{L/K}^\ab(t)= \psi^\ab_{\lambda,i}(t) + 1/p^i$ for $i$ sufficiently large or   $\psi_{L/K}^\ab(t)= \psi^\ab_{\pi}(t)$, in the notation of \autoref{theorem:piecewiselinear}. If
	 $\psi_{L/K}^\ab(t)= \psi^\ab_{\pi}(t)$, then the slope of  $\psi_{L/K}^\ab$ is an integer on this interval, so there is nothing to prove. On the other hand, if 
	 \[
	    \psi_{L/K}^\ab(t)= \psi^\ab_{\lambda,i}(t) + 1/p^i = \dfrac{\psi^\ab_{\widehat{L_{\lambda,i}}/K}(t)}{p^i} + \dfrac{1}{p^i},
	 \]
	 then it is sufficient to prove that the slope of $\psi^\ab_{\widehat{L_{\lambda,i}}/K}(t)$ on this interval is an integer divisible by $p^i$.  
	 
	 For $e\in \Z$ sufficiently large and prime to $p$, we can find $t_1, t_2\in I$ such that $t_1 = \frac{p^im_1}{e}$, where $m_1\in \Z$, and $t_2 =   \frac{m_2}{e}$, where $m_2\in \Z$ and $(m_2,p)=1$. Let $K'/K$ be a finite, totally tamely ramified extension with ramification index $e$, and write $L'=LK'$, $\widehat{L_{\lambda,i}}'=\widehat{L_{\lambda,i}}K'$. Then, for $t\in I$, we have 
	 \[
		  \dfrac{\psi_{L'/K'}^\ab(et)}{e(L'/L)}= \dfrac{\psi^\ab_{\widehat{L_{\lambda,i}}'/K'}(et)}{e(L'/L)p^i} + \dfrac{1}{p^i},
	 \]  
	 so
	 \[
		 \psi_{L'/K'}^\ab(et)= \dfrac{1}{p^i}\left(\psi^\ab_{\widehat{L_{\lambda,i}}'/K'}(et) + e(L'/L)\right).
	 \]  
	 For $t\in I$,  
	 \[
		 \psi^\ab_{\widehat{L_{\lambda,i}}'/K'}(et) = a et +b
	 \]
	 for some $a\in \Z, b\in \Q$. We will show that $p^i\mid a$. From (i),
	 \[
		 \Z \ni \psi_{L'/K'}^\ab(et_1)= \dfrac{1}{p^i}\left(aet_1+b + e(L'/L)\right)=am_1+\dfrac{b+e(L'/L)}{p^i},
	 \]  
	 so $\frac{b+e(L'/L)}{p^i}\in \Z$. In the other hand, 
	 \[
		  \Z \ni \psi_{L'/K'}^\ab(et_2)= \dfrac{1}{p^i}\left(am_2+b + e(L'/L)\right)=\dfrac{am_2}{p^i}+\dfrac{b+e(L'/L)}{p^i},
	 \]  
	 so $\frac{am_2}{p^i}\in \Z$. Since $(m_2, p)=1$, we have $p^i\mid a$. Further, for $t\in I$, 
	 \[
		 \psi_{L/K}^\ab(t)= \dfrac{\psi_{L'/K'}^\ab(et)}{e(L'/L)} = \dfrac{aet+b + e(L'/L)}{p^i e(L'/L)}.
	 \]
	 Since $p^i\mid a$, $e(L'/L)$ is prime to $p$, and 
	 \[
	 \psi_{L/K}^\ab(t)= \dfrac{\psi^\ab_{\widehat{L_{\lambda,i}}/K}(t)}{p^i} + \dfrac{1}{p^i},
	 \]
	  we have that $\frac{ae}{p^i e(L'/L)} \in \Z$. Thus the slope of $\psi_{L/K}^\ab$ on this interval is an integer. Finally, to prove (ii), consider again an interval $I$ where  $\psi_{L/K}^\ab$ is linear and  $\psi_{L/K}^\ab(t)= \psi^\ab_{\lambda,i}(t) + 1/p^i$ or   $\psi_{L/K}^\ab(t)= \psi^\ab_{\pi}(t)$. From the fact that $\psi^\ab_{\lambda,i}(t)$ and  $\psi^\ab_{\pi}(t)$ have rational constant terms, the result follows. 
%	  On $I$, 
%	 \[
%	 \psi^\ab_{\widehat{L_{\lambda,1}}/K}(t) = a t +b
%	 \]
%	 for some $a\in \Z, b\in \Q$, so, for $t\in I$, we have
%	 \[
%	 \psi^\ab_{\widehat{L_{\lambda,1}}'/K'}(et) =e(L'/L) a t + e(L'/L)b.
%	 \] 
\end{proof}

\section{\texorpdfstring{Computation of $\psi_{L/K}^{\ab}$ for basic cases}{Computation of psi_{L/K}^ab for basic cases}}

Through this section, $K$ shall denote a complete discrete valuation field of  characteristic $p>0$ with perfect residue field, and $L_0/K$ an extension of complete discrete valuation fields such that $e(L_0/K)=1$. The residue field of $L_0$ is not assumed to be perfect.

In this section we will compute  $\psi_{L/K}^{\ab}$ in some cases.  
For a separable extension of complete discrete valuation fields $L/K$, and $a\in K$, define %separable is need for inj on diff
\[
\delta_{L/K}(a) := -v^{\log}_{K}(da) +\dfrac{v^{\log}_{L}(da)}{e(L/K)}.
\] 
Intuitively, we can see $\delta_{L/K}(a)$ as $\delta_{L/K}(a)=\pole_K(da) -\pole_L(da)$. Whenever there is no ambiguity, we write simply $\delta(a)$.

\subsection{Tamely ramified extensions}

In this subsection, we consider the simplest case, that of an extension of complete discrete valuation fields $L/K$ such that $e(L/K)$ is prime to $p$ (we say that such extension is tamely ramified). More generally, we consider extensions satisfying $\delta_\tor(L/K)=0$. In this case, from the formula obtained in \autoref{theorem:oldformula}, we have 
\[\psi_{L/K}^\ab(t)=e(L/K)t. \]

The formula for $\psi_{L/K}^\ab(t)$ for a tamely ramified extension of complete discrete valuation fields $L/K$ is an immediate consequence of the fact that $\delta_\tor(L/K)=0$ in this case. To show that $\delta_\tor(L/K)=0$, denote $e(L/K)=e$. We have $\pi_K = u\pi^e_L$ for some $u\in U_L$, and \begin{align*}\delta_\tor(L/K)&= v_L^\log\left(\frac{d\pi_K}{\pi_K}\right)\\&= v_L^\log\left(\frac{du}{u} + e\frac{d\pi_L}{\pi_L}\right)\\&=0.\end{align*}
Therefore, $\delta_\tor(L/K)=0$, so it immediately follows that  
\[\psi_{L/K}^\ab(t)=e(L/K)t. \]
\subsection{\texorpdfstring{Totally ramified cyclic extension $L/L_0$ of degree $p$}{Totally ramified cyclic extension L/L\_0 of degree p}}

In this subsection, $L/L_0$ denotes a  separable, totally ramified cyclic extension of degree $p$. Recall that, when $K_1/K$ is a separable, totally ramified cyclic extension of degree $p$, the classical Hasse-Herbrand $\psi$-function takes the following form (\cite[Chapter V, \S 3]{serre1979local}):
\[
\psi_{K_1/K}(t) = \begin{cases}
t, & t \leq \dfrac{\delta_{\tor}(K_1/K)}{p-1} \\ pt - \delta_\tor(K_1/K), & t> \dfrac{\delta_{\tor}(K_1/K)}{p-1}
\end{cases}\]
 We will show that  $\psi^{\ab}_{L/K}$ is similar to $\psi_{K_1/K}$.

\begin{lemma} \label{lemma:degreep}Let $\chi \in H^1(K, \Z/p\Z)$, and $a\in K$ correspond to $\chi$ under Artin-Schreier-Witt theory. Assume that $a$ is best in $K$. Then  
	\[
	\sw \chi_L = \begin{cases} \sw \chi, &   - v^{\log}_L(da) <  p  \dfrac{\delta(a)}{p-1}\\
	p(\sw \chi - \delta(a)), & - v^{\log}_L(da) >  p  \dfrac{\delta(a)}{p-1}
	\end{cases}
	\]
\end{lemma}
\begin{proof}

	If $v_K(a)\geq 0$ the result is clear, so assume $v_K(a)<0$. Since $a\in K$ is best, we have that $\sw \chi = - v_K(a)$. Further, $p\nmid v_K(a)$, and $a$ is also best in $L_0$, so \[\sw \chi_{L_0}= - v_{L_0}(a)=-v^{\log}_{L_0}(da).\]

	If $a$ is best in $L$, then $\delta(a)=0$ and $\sw \chi_L = p \sw \chi$, so the formula holds. Assume that $a$ is not best in $L$. Write \[a= b_0^p - b_0 + a_0,\] where $a_0$ is best in $L$ (recall that we have $v_L(a)< v_L(a_0)$), and put $c_0= -b_0+ a_0$. We have \[a = b_0^p +c_0.\] 
	
	If $c_0$ is best in $L$, we put $b=b_0$ and $c=c_0$ to write $a=b^p+c$ with $c$ best. If not,  observe that $v_L(a)< \min \{v_L(b_0),v_L(a_0) \}\leq v_L(c_0)$. Repeating the argument, write $c_0 = b_1^p + c_1$ with $v_L(c_0) < v_L(c_1)$, so that $a= (b_0+b_1)^p+c_1$. By induction on $v_L(c_i)$, this process eventually stops, so we can write 
	$$a=b^p + c$$
	for some $b, c\in L$ with $c$ best.

	Observe that $v_L(b) = p^{-1}v_L(a)=v_K(a)$, and, since $da=dc$ and $c$ is best in $L$, we have $v^{\log}_L(da)=v_L^{\log}(dc)=v_L(c)$. Consider the following two cases:
	\begin{enumerate}[(i)]
		\item  Assume $v_L(b)<v_L(c)$. This condition can be rewritten as $v_K(a)<v_L^{\log}(da)$, or 
		\[-\dfrac{p-1}{p} v^{\log}_L(da)<\delta(a).\]

		Write $a= b^p - b + (b+c)$. We claim that $b$ is best. Since $a\in K \subset L_0$ is best in $K$, $p\nmid v_K(a)=v_{L_0}(a)$. Since the extension is of degree $p$ and totally ramified, $v_L(a)= p v_K(a)$, so $p\mid v_L(a)$ but $p^2 \nmid v_L(a)$. Recall that $v_L(a) = p v_L(b)$. 		
		If $b$ were not best, then we would have $p\mid v_L(b)$. But then we would have $p^2\mid v_L(a)$, a contradiction. 
	
		Then $b+c$ is best, and $\sw \chi_L = - v_L(b+c) = - v_L(b) = - v_K(a)$,  so \[\sw \chi_L= \sw \chi.\]

		\item Assume $v_L(b)>v_L(c)$. This condition can be rewritten as  \[-\dfrac{p-1}{p} v^{\log}_L(da)>\delta(a).\]
		
		In this case, we have $a = b^p -b +(b+c)$, with $b+c$ best and\[v_L(b+c) = v_L(c) = v_L^{\log}(da) = p(\delta(a)+v_L(a)),\] so \[\sw \chi_L  = p \sw \chi - p\delta(a). \qedhere\]  
		\end{enumerate}
\end{proof}

\begin{remark}
	Observe that, for $a$ as in \autoref{lemma:degreep},   if $- v_K(a) <    \dfrac{\delta(a)}{p-1}$, then we also have 
	$- v^\log_L(da) <  p  \dfrac{\delta(a)}{p-1}$.
\end{remark}

\begin{proposition} \label{proposition:degreep} Assume that $\delta(L/K)\neq 0$. Then there exists a finite, tame extension $K'/K$ and $\chi_{K'}\in H^1(K', \Z/p\Z)$ such that \[\sw \chi_{L'} = \sw \chi_{K'}>0,\] where $L'=LK'$. Further, $\chi_{K'}$ can be taken to satisfy  $0<\sw \chi_{K'}<\frac{\delta_{L'/K'}(a)}{p-1}$, where $\tilde{a}\in K'$ is best and corresponds to $\chi_{K'}$.   
\end{proposition}
\begin{proof} Take $a\in K$ such that $a$ is best, $v_K(a)<0$, and $\delta_{L/K}(a)\neq 0$. Such $a$ exists since $\delta(L/K)\neq 0$. If 
	 $- v_K(a) <    \dfrac{\delta_{L/K}(a)}{p-1}$, then, by \autoref{lemma:degreep}, it correspond to a $\chi \in H^1(K, \Z/p\Z)$ such that $\sw \chi_L = \sw \chi$. If not, let $e,r \in \Z_{>0}$ be such that $(e,p)=1$ and \[0<-e\left(v_K(a) + \dfrac{\delta_{L/K}(a)}{p-1} \right)<p^r <- e v_K(a).\] 
	 This guarantees that $-ev_K(a)-p^r >0$ and $-ev_K(a)-p^r <  e\dfrac{\delta_{L/K}(a)}{p-1}$. 
	 Let $K'=K(\sqrt[e]{\pi_K})$. It is a totally ramified tame extension with ramification index $e$. Putting $L_0'=L_0K'$ and $L=LK'$, we have that $e(L_0'/K')=1$ and $L'/L_0'$ is a totally ramified cyclic extension of degree $p$. We have $-v_{K'}(a)-p^r>0$ and $-v_{K'}(a)-p^r< \dfrac{\delta_{L'/K'}(a)}{p-1}$.  
	 
	 Let $m=p^r$ and $\tilde{a} = \pi_{K'}^ma$. Observe that  \begin{align*}\delta_{L'/K'}(\tilde{a})&=-v^{\log}_{K'}(\pi_{K'}^mda) +\dfrac{v^{\log}_{L'}(\pi_{K'}^mda)}{e(L'/K')}\\ &= -v^{\log}_{K'}(da)-m +\dfrac{v^{\log}_{L'}(da)}{e(L'/K')} +m \\ &=\delta_{L'/K'}(a).\end{align*}
	 On the other hand, since $p\nmid v_K(a)$, we also have $p\nmid v_{K'}(\tilde{a})$. Then $\tilde{a}$ is best in $K$, and 	 $v_{K'}(\tilde{a}) = m+v_{K'}(a)$. Thus $v_{K'}(\tilde{a})<0$, and \[0<-v_{K'}(\tilde{a})< \dfrac{\delta_{L'/K'}(\tilde{a})}{p-1}.\]  
	 Let $\chi_{K'}\in H^1(K',\Z/p\Z)$ be the character corresponding to $\tilde{a}$ under Artin-Schreier-Witt theory. Then, by \autoref{lemma:degreep},
	 \[\sw \chi_{L'}= \sw \chi_{K'}.\qedhere\] 
\end{proof}

\begin{theorem}\label{theorem:degreep} Assume that $L/L_0$ is a separable, totally ramified cyclic extension of degree $p$. Then
	\[
	\psi^{\ab}_{L/K}(t) = \begin{cases}
		t, & t \leq \dfrac{\delta_{\tor}(L/K)}{p-1} \\ pt - \delta_\tor(L/K), & t> \dfrac{\delta_{\tor}(L/K)}{p-1}
	\end{cases}
	\]
\end{theorem}
\begin{proof} From \autoref{theorem:oldformula} and the continuity of $\psi_{L/K}^{\ab}$ (\autoref{theorem:cont}), we have that 
	$\psi_{L/K}^{\ab}(t)=  pt - \delta_\tor(L/K)$ for $t\geq \dfrac{\delta_\tor(L/K)}{p-1}$. Observe also that $\psi_{L/K}^{\ab}(0)=0$, and 
	\[
		\psi_{L/K}^{\ab}\left(\frac{\delta_\tor(L/K)}{p-1}\right) = p \frac{\delta_\tor(L/K)}{p-1} - \delta_\tor(L/K) = \frac{\delta_\tor(L/K)}{p-1}.
	\]
	Since $\psi^{\ab}_{L/K}$ is convex (\autoref{theorem:cont}), it follows that  $0\leq\psi_{L/K}^{\ab}(t)\leq t$ for $0\leq t \leq \frac{\delta_\tor(L/K)}{p-1}$. It suffices to show that $\psi_{L/K}^{\ab}(t_0)\geq t_0$ for some 
	$0< t_0 < \frac{\delta_\tor(L/K)}{p-1}$. Indeed, if $\psi_{L/K}^{\ab}(t_0)\geq t_0$, then $\psi_{L/K}^{\ab}(t_0)= t_0$. If $0<\tilde{t}< t_0$ is such that $\psi_{L/K}^{\ab}(\tilde{t})<\tilde{t}$, then, by convexity, $\psi_{L/K}^{\ab}(t)< t$ for $\tilde{t}<t < \frac{\delta_\tor(L/K)}{p-1}$. This contradicts $\psi_{L/K}^{\ab}(t_0)\geq t_0$. A similar argument can be made to show that $\psi_{L/K}^{\ab}(\tilde{t})\geq\tilde{t}$ for $t_0<\tilde{t}< \frac{\delta_\tor(L/K)}{p-1}$.

	This shows that it is  only necessary to prove that $\psi_{L/K}^{\ab}(t_0)\geq t_0$ for some  $0< t_0 < \frac{\delta_\tor(L/K)}{p-1}$, which follows from \autoref{proposition:degreep} and the definition of $\psi_{L/K}^{\ab}$.
 \end{proof}

\begin{ack}
	I would like to deeply thank Professor Kazuya Kato for kindly sharing his helpful feedback and advice throughout the development of this paper.    
\end{ack}

\bibliographystyle{amsplain}
\bibliography{bibfile}

\end{document}